\documentclass[a4paper,12pt,reqno]{amsart}
\usepackage[text={130mm,200mm}]{geometry}

\theoremstyle{plain}
\newtheorem{theorem}{Theorem}
\newtheorem*{theorem*}{Theorem}
\newtheorem{corollary}{Corollary}
\newtheorem*{corollary*}{Corollary}
\newtheorem{lemma}{Lemma}
\newtheorem*{lemma*}{Lemma}
\newtheorem{proposition}{Proposition}
\newtheorem*{proposition*}{Proposition}

\newtheorem*{conjecture*}{Conjecture}
\theoremstyle{definition}
\newtheorem{definition}{Definition}
\newtheorem*{definition*}{Definition}
\theoremstyle{remark}
\newtheorem{remark}{Remark}
\newtheorem*{remark*}{Remark}

\begin{document}

\title[Application of infinite systems of functional equations]{On one class of functions with complicated local structure that the solutions of infinite systems of functional equations (On one application of infinite systems of functional equations in the functions theory)}
\author{Symon Serbenyuk}
\address{
  45, Shchukina Str. \\
  Vinnytsia \\
  21012 \\
  Ukraine}
\email{simon6@ukr.net}



\begin{abstract}

The article is devoted to investigation of applications of infinite systems of functional equations for modeling of functions with complicated local structure, that defined in terms of nega-$\tilde Q$-representation. The infinite   system of functional equations
$$
f\left(\hat \varphi^k(x)\right)=\tilde \beta_{i_{k+1},k+1}+\tilde p_{i_{k+1},k+1}f\left(\hat \varphi^{k+1}(x)\right),
$$
where  $k=0,1,...$, $\hat \varphi$ is a shift operator of  $\tilde Q$-expansion, $x=\Delta^{-\tilde Q} _{i_1(x)i_2(x)...i_n(x)...}$, are investigated. It is proved, that the system has an  unique solution in the class of determined and bounded on $[0;1]$ functions  and continuity of the solution. His analytical presentation is founded. Conditions of its monotonicity and nonmonotonicity,  differential, integral properties are studied. Conditions under which
the solution of the functional equations system is a distribution function of random variable $\eta=\Delta^{\tilde Q} _{\xi_1\xi_2...\xi_n...}$ with independent $\tilde Q$-symbols are  discovered.
\end{abstract}

\maketitle


\section{Introduction}

Nowadays it is well known that functional equations and systems of functional equations are using widely in mathematics and other sciences. For example, in the information theory, physics, economics, decision theory,  etc.  \cite{A1988, AD2003, D1985,  L1997}. A modeling of functions with complicated local structure by systems of functional equations is a shining example of their applications in the function theory. 

A class of functions with complicated local structure consists of singular (for example, \cite{{Salem1943}, {S. Serbenyuk abstract 1}, {S. Serbenyuk abstract 1-article}, {Zamfirescu1981}, {Pra92}}), continuous nowhere monotonic \cite{PK2013} and nowhere differentiable functions  (for example, \cite{{Bush1952}, {Symon12(2)}, {S. Serbenyuk abstract 8}, {S. Serbenyuk, functions with complicated local structure 2013}, {S. Serbenyuk abstract 6},{S. Serbenyuk systemy rivnyan 2-2},  {Pra92}}). 

An example of a strictly  increasing singular function  is the following function, that is called the Salem function
$$
f(x)=\beta_{\alpha_1(x)}+\sum^{\infty} _{n=1}{\left(\beta_{\alpha_n(x)}\prod^{n-1} _{j=1}{p_{\alpha_j(x)}}\right)},
$$
where 
$$
x=\Delta^2 _{\alpha_1\alpha_2...\alpha_n...}\equiv \sum^{\infty} _{n=1}{\frac{\alpha_n}{2^n}}, ~\alpha_n\in\{0,1\},
$$
 $Q_2=\{p_0,p_1\}$ is a fixed tuple of integers such that  $p_0+p_1=1$, and
$$
\beta_{\alpha_n(x)}=\begin{cases}
0,&\text{if $\alpha_n(x)=0$,}\\
p_0,&\text{if $\alpha_k(x)=1$,}
\end{cases}
$$
Properties of the function (including singularity) were studied by Salem in \cite{Salem1943} and other authors  in  \cite{M1971, Pra92, K2008}. The Salem function is a distribution function of  a random variable with  independent identically distributed binary digits and an unique solution of the following functional equations system in the class of determined and bounded on $[0;1]$ functions:
$$
\left\{
\begin{aligned}
f\left(\frac{x}{2}\right)&=p_0f(x),\\
f\left(\frac{x+1}{2}\right) & = p_0+p_1f(x).\\
\end{aligned}
\right.
$$
The system can be   written as follows (for functions, that determined on the segment $[0;1]$): 
$$
f\left(\Delta^2 _{i\alpha_1\alpha_2...\alpha_n...}\right)=\beta_i+p_if\left(\Delta^2 _{\alpha_1\alpha_2...\alpha_n...}\right).
$$

In \cite{PK2011} M. Pratsiovytyi and A. Kalashnikov investigated the following generalization of the Salem function 
$$
f(x)=\beta_{\alpha_1(x)}+\sum^{\infty} _{n=2}{\left(\beta_{\alpha_n(x)}\prod^{n-1} _{j=1}{p_{\alpha_j(x)}}\right)},
$$
where
$$
x=\Delta^s _{\alpha_1\alpha_2...\alpha_n...}\equiv\sum^{\infty} _{n=1}{\frac{\alpha_n}{s^n}}, ~~~\alpha_n\in\{0,1,...,s-1\},
$$
$2<s$ is a fixed positive integer number, 
$$
\beta_{\alpha_n(x)}=\begin{cases}
0,&\text{if $\alpha_n(x)=0$,}\\
\sum^{\alpha_n(x)-1} _{i=0}{p_i}>0,&\text{if $\alpha_n(x)\ne 0$,}
\end{cases}
$$
The last-mentioned function is an unique solution of the following system of functional equations  in the class of determined and bounded on $[0;1]$ functions:
$$
f\left(\frac{i+x}{s}\right)= \beta_i+p_if(x),
$$
where $i=0,1,...,s-1$,  $p_i$  is a real numbers from  $(-1;1)$ and $p_0+$ $+p_1~+~...~+p_{s-1}=1$.  

In \cite{PK2013} these authors considered the system of $s$ functional equations 
$$
f\left(\Delta^Q _{i\alpha_1\alpha_2...\alpha_n...}\right)= \beta_{i}+p_{i}f\left(\Delta^Q _{\alpha_1\alpha_2...\alpha_n...}\right),~~~i=\overline{0,s-1},
$$
where $\max_i {|p_i|}<1$, $p_0+p_1+...+p_{s-1}=1$, $\beta_0=0<\beta_k=\sum^{k-1} _{i=0}{p_i}$, 
and an argument of $f$ is represented in terms a $Q$-representation \cite[p.~87]{Pra98}. A $Q$-representation is a generalization of an s-adic representation $\Delta^s _{\alpha_1\alpha_2...\alpha_n...}$. That is 
$$
\Delta^Q _{\alpha_1\alpha_2...\alpha_n...}\equiv \gamma_{\alpha_1}+\sum^{\infty} _{n=2}{\left(\gamma_{\alpha_n}\prod^{n-1} _{j=1}{q_{\alpha_j}}\right)}=x\in [0;1],
$$
where $1<s$ is a fixed positive integer number, $Q=\{q_0,q_1,...,q_{s-1}\}$ is a fixed tuple of real numbers, that $q_i>0$ for all $i=\overline{0,s-1}$, $\sum^{s-1} _{j=0}{q_j}=1$, $\gamma_0=0$, $\gamma_k=\sum^{k-1} _{j=0}{q_j}$ for $k=1,2,...,s-1$.
The investigations from  \cite{PK2013} are generalization of investigations  from \cite{PK2011}.

Introducing and investigations of generalizations of the Salem function are new and non-known for real numbers representations $\Delta_{\delta_1\delta_2...\delta_n...}$ with removable  alphabet. That is, when  $\delta_n\in A^{0} _{t_n}\equiv\{0,1,...,t_n\}$, $t_n\in\mathbb N$, and $|A^{0} _{t_k}|\ne~|A^{0} _{t_l}|$ for $k\ne l$, where $|\cdot|$ is  the number of elements of the set. Representations of real numbers by the positive \cite{{Cantor1}, {Ralko10},{Symon3}} and the alternating \cite{{S. Serbenyuk, alternating Cantor series, 2013}, {Symon4}}  Cantor series, $\tilde Q$- \cite[p. 87-91]{Pra98} and nega-$\tilde Q$-representation \cite{{S. Serbenyuk abstract 10}, {S. Serbenyuk nega-tilde Q-representation}} are these representations. For the first time such investigations were carried out exactly by the author of the present article for the case of the positive Cantor series \cite{{S. Serbenyuk abstract 9}, {S. Serbenyuk systemy rivnyan 1}} and presented in April 2014 on International Mathematical Conference "Differential equations, computational mathematics, functions  theory and mathematical methods in mechanics"~\cite{{S. Serbenyuk abstract 9}}.

In October 2014 the contents of this article and similar researches for the case of the positive and alternating Cantor series (the papers \cite{{S. Serbenyuk abstract 9}, {S. Serbenyuk systemy rivnyan 1}, {S. Serbenyuk systemy rivnyan 2}, {S. Serbenyuk abstract 10}, {S. Serbenyuk nega-tilde Q-representation}, {S. Serbenyuk abstract 11}}) were presented by the author in the reports "Determination of a class of functions, that represented by the Cantor series, by systems of functional equations", "Polybasic positive and alternating $\tilde Q$-representations and their applications to determination of functions by systems of  functional equations" at a seminar on fractal analysis  of Institute of Mathematics of NAS of Ukraine and National Pedagogical Dragomanov University (archive of reports is available here: http://www.imath.kiev.ua/events/index.php?seminarId=21\&archiv=1). 

The present article is devoted to studying of one example of applications of systems of functional equations  to modeling of functions with complicated local structure. The following  functional equations system, that is a new and non-investigated at this moment, is studied:
$$
f\left(\hat \varphi^k(x)\right)=\tilde \beta_{i_{k+1},k+1}+\tilde p_{i_{k+1},k+1}f\left(\hat \varphi^{k+1}(x)\right),
$$
where $k=0,1,...$, $\hat \varphi$ is a shift operator of a  $\tilde Q$-expansion, $x=\Delta^{-\tilde Q} _{i_1(x)i_2(x)...i_n(x)...}\equiv \ \equiv \Delta^{\tilde Q} _{i_1(x)[m_2-i_2(x)]i_3...i_{2k-1}(x)[m_{2k}-i_{2k}(x)]...}$.
The properties of the unique solution of the last-mentioned system in the class of determined and bounded on [0; 1] functions 
$$
F(x)=\beta_{i_1(x),1}+\sum^{\infty} _{k=2}{\left[\tilde \beta_{i_k(x),k}\prod^{k-1} _{j=1}{\tilde p_{i_j(x),j}}\right]}.
$$
are investigated.

\section{$\tilde Q$-representation, its partial cases and shift operator}

Let $\tilde Q=||q_{i,n}||$ be a fixed matrix, where  $i=\overline{0,m_n}$, $m_n \in N^{0} _{\infty}= \mathbb N \cup \{0,\infty\}$, $n=1,2,...$, and the following system of poperties is true for elements $q_{i,n}$ of the last-mentioned  matrix
$$
\left\{
\begin{aligned}
\label{eq: tilde Q 1}
1^{\circ}.  ~~~~~~~~~~~~~~~~~~~~~~~~~~~~~~~~q_{i,n}>0;\\
2^{\circ}.  ~~~~~~~~~~~~~~~\forall n \in \mathbb N: \sum^{m_n}_{i=0} {q_{i,n}}=1;\\
3^{\circ}.  \forall (i_n), i_n \in \mathbb N \cup \{0\}: \prod^{\infty} _{n=1} {q_{i_n,n}}=0.\\
\end{aligned}
\right.
$$

\begin{definition}
An expansion of a number  $x$ from $[0;1)$ by the following positive series 
\begin{equation}
\label{eq: tilde Q-representation1}
a_{i_1(x),1}+\sum^{\infty} _{n=2} {\left[a_{i_n(x),n}\prod^{n-1} _{j=1} {q_{i_j(x),j}} \right]},
\end{equation}
where
$$
a_{i_n,n}=\begin{cases}
\sum^{i_n-1} _{i=0} {q_{i,n}},&\text{if $i_n \ne 0$,}\\
0,&\text{if $i_n=0$,}
\end{cases}
$$
is called  \emph{a $\tilde Q$-expansion}  \emph{of a number $x$} ~{\cite[p.~89]{Pra98}}.

A defining of an arbitrary number $x\in[0;1)$ by the expansion  \eqref{eq: tilde Q-representation1} is denoted by  $x=\Delta^{\tilde Q} _{i_1i_2...i_n...}$ and the last-mentioned notation is called  \emph{a $\tilde Q$-representation of $x$.}
\end{definition}

\begin{definition}
A number  $x \in [0;1)$ which has  period (0) in its $\tilde Q$-representation   is called   {a $\tilde Q$-rational number}, i. e.
$\Delta^{\tilde Q} _{i_1i_2...i_{n-1}i_n(0)}$. The other numbers from $[0;1]$ are called   {$\tilde Q$-irrational numbers}.
\end{definition}

\begin{proposition}
If $m_n<\infty$ for all $n> n_0$, where $n_0$ is a fixed integer, then each $\tilde Q$-rational number has two different $\tilde Q$-representations, i. e. 
$$
\Delta^{\tilde Q} _{i_1i_2...i_{n-1}i_n(0)}=\Delta^{\tilde Q} _{i_1i_2...i_{n-1}[i_n-1]m_{n+1}m_{n+2}...}.
$$
\end{proposition}

A $\tilde Q$-representation of real numbers is:
\begin{itemize}
\item the $Q^{*}$-representation (or  $Q^{*} _{s}$-representation), when $m_n=s-1$ for all  $n\in\mathbb N$, where  $\mathbb N\ni s=const>1$;
\item the $Q$-representation (or  $Q_{s}$-representation), when $m_n=s-1$ for arbitrary  $n\in\mathbb N$, where $\mathbb N\ni s=const>1$ and $q_{i,n}=q_i$ for all $n\in\mathbb N$;
\item the $Q^{*} _{\infty}$-representation, if  $m_n=\infty$ for any  $n\in\mathbb N$;
\item the $Q_{\infty}$-representation, if  $m_n=\infty$ and  $q_{i,n}=q_i$ for all  $n\in\mathbb N$;
\item  a representation by the positive Cantor series, if  $m_n=d_n-1$ and $q_{i,n}=\frac{1}{d_n}$, $i=\overline{0,d_n-1}$, for any   $n\in\mathbb N$, where $(d_n)$ is a fixed sequence of  positive integers $d_n>1$;
\item an  s-adic representation, if $m_n=s-1$ and $q_{i,n}=q_{i}=\frac{1}{s}$ for all  $n\in\mathbb N$, where $s>1$ is a fixed positive integer. 
\end{itemize}

\begin{definition} The mapping is defined by
$$
\hat \varphi(x)=\hat \varphi \left(\Delta^{\tilde Q} _{i_1i_2...i_n...}\right)=a_{i_2,2}+\sum^{\infty} _{k=3}{\left[a_{i_k,k}\prod^{k-1} _{j=2}{q_{i_j,j}}\right]}
$$
is called \emph{the shift operator $\hat \varphi$ of the $\tilde Q$-expansion of a number $x$}. 

The following mapping
$$
\hat \varphi^k(x)=a_{i_{k+1},k+1}+\sum^{\infty} _{n=k+2}{\left[a_{i_n,n}\prod^{n-1} _{j=k+1}{q_{i_j,j}}\right]}
$$
is called \emph{the shift operator of rank $k$ of the $\tilde Q$-expansion of a number $x$}.
\end{definition}

The last-mentioned and the following definitions of $\hat \varphi^k$ are equivalents. 
\begin{definition}
Let $\left(\tilde Q_k\right)$  be a sequence of the following matrixes  $\tilde Q_k$:
$$
\tilde Q_k=
\begin{pmatrix}
q_{0,k+1}& q_{0,k+2} &\ldots & q_{0,k+j}&\ldots\\
q_{1,k+1}&  q_{1,k+2} &\ldots & q_{1,k+2}&\ldots\\
\vdots& \vdots &\ddots & \vdots&\ldots\\
q_{m_{k+1}-1,k+1}& q_{m_{k+2}-2,k+2} &\ldots & q_{m_{k+j},k+j} &\ldots\\
q_{m_{k+1},k+1}  &    q_{m_{k+2}-1,k+2}        &...&                        & ...\\
                         &  q_{m_{k+2},k+2}                       & ...&                & ...
\end{pmatrix},
$$
 where $k=0,1,2,...,$ $j=1,2,3,...$.

Let $\left(\mathcal F^{\tilde Q_k} _{[0;1)}\right)$ be a sequence of sets $\mathcal F^{\tilde Q_k} _{[0;1)}$ of all possible   $\tilde Q_k$-representations, that generated by the matrix  $\tilde Q_k$, of numbers from   $[0;1)$, where $\tilde Q_0\equiv \tilde Q$.

 The  mapping  $\pi\left(\hat \varphi^k (x,\tilde Q)\right)$ such that
$$
\varphi^k: [0;1)\times \tilde Q \rightarrow [0;1)\times \tilde Q_k
 $$
$$
((i_1,i_2,...,i_k,...),\tilde Q ) \rightarrow ((i_{k+1},i_{k+2},i_{k+j},...),\tilde Q_k),
$$ 
$$
\pi: [0;1)\times \tilde Q_k \rightarrow [0;1) 
$$
$$
(x^{'},\tilde Q_k) \rightarrow x^{'}.
$$
is called \emph{the shift  operator $\hat \varphi^k$ of rank $k$ of the $\tilde Q$-representation of a number  $x\in~[0;1)$}.
\end{definition}

\begin{remark}
For compactness of the presentation paper  we shall use the notation $\hat \varphi^k$ instead  the notation $\pi\left(\hat \varphi^k (x,\tilde Q)\right)$ hereinafter. 
\end{remark}

Since  $x=a_{i_1,1}+q_{i_1,1}\hat \varphi(x)$, 
$$
\hat \varphi(x)=\frac{x-a_{i_1,1}}{q_{i_1,1}}.
$$

It is easy to see that the following expressions are true. 
$$
\hat \varphi^k(x)=\frac{1}{q_{0,1}q_{0,2}...q_{0,k}}\Delta^{\tilde Q} _{\underbrace{0...0}_{k}i_{k+1}i_{k+2}...},
$$
\begin{equation}
\label{shift operator-tilde Q}
\hat \varphi^{k-1}(x)=a_{i_k,k}+q_{i_k,k}\hat \varphi^k(x)=\frac{1}{q_{0,1}q_{0,2}...q_{0,k-1}}\Delta^{\tilde Q} _{\underbrace{0...0}_{k-1}i_{k}i_{k+1}...}.
\end{equation}

\begin{remark*}
Hereinafter at study we shall consider that $m_n<\infty$ for all $n\in~\mathbb N$.
\end{remark*}

\section{Nega-$\tilde Q$-representation of real numbers from $[0;1]$}

\begin{theorem}{\cite{S. Serbenyuk nega-tilde Q-representation, {S. Serbenyuk abstract 11}}.}
For arbitrary $x\in [0;1)$ there exists a sequence $(i_n)$,  $i_n\in N^{0} _{m_n}$, such that  
\begin{equation}
\label{def: nega-tilde Q 3}
x= \sum^{i_1-1} _{i=0}{q_{i,1}}+\sum^{\infty} _{n=2}{\left[(-1)^{n-1}\tilde \delta_{i_n,n}\prod^{n-1} _{j=1}{\tilde q_{i_j,j}}\right]}+\sum^{\infty} _{n=1}{\left(\prod^{2n-1} _{j=1}{\tilde q_{i_j,j}}\right)},
\end{equation}
where
$$
\tilde \delta_{i_{n},n}=\begin{cases}
1,&\text{if $n$ is an even number  and $i_{n}=m_{n}$;}\\
\sum^{m_{n}} _{i=m_{n}-i_{n}} {q_{i,n}},&\text{if $n$ is an even number  and $i_{n}\ne m_{n}$;}\\
0,&\text{if $n$ is an odd number  and $i_{n}=0$;}\\
\sum^{i_n-1} _{i=0}{q_{i,n}},&\text{if $n$ is an odd number  and $i_{n}\ne0$,}\\
\end{cases}
$$
and the first sum in the expression  \eqref{def: nega-tilde Q 3} equals to $0$, if  $i_1=0$.  
\end{theorem}

\begin{definition}
An expansion of a number  $x$ by the series \eqref{def: nega-tilde Q 3} is called  \emph{a nega-$\tilde Q$-expansion of  $x$} and denotes by  $\Delta^{-\tilde Q} _{i_1i_2...i_n...}$. The last-mentioned notation is called a  \emph{nega-$\tilde Q$-representation of a number  $x$} \cite{S. Serbenyuk nega-tilde Q-representation}.
\end{definition}

Numbers from a some countable subset of $[0;1]$ have two different nega-$\tilde Q$-representations, i. e. 
$$
\Delta^{-\tilde Q} _{i_1i_2...i_{n-1}i_nm_{n+1}0m_{n+3}0m_{n+5}...}=\Delta^{-\tilde Q} _{i_1i_2...i_{n-1}[i_n-1]0m_{n+2}0m_{n+4}...}, ~i_n\ne 0.
$$
These numbers are called  \emph{nega-$\tilde Q$-rationals} and the rest of the numbers from $[0;1]$ are called  \emph{nega-$\tilde Q$-irrationals}.

\begin{definition}
A set of all numbers from $[0;1]$ such that the first $n$ digits $i_1,i_2,...,i_n$ of the nega-$\tilde Q$-representation of the numbers are equal to   $c_1,c_2,...,c_n$ respectively  is called  \emph{a cylinder $\Delta^{-\tilde Q} _{c_1c_2...c_n}$ of rank   $n$ with the base $c_1c_2...c_n$}. That is
$$
\Delta^{-\tilde Q} _{c_1c_2...c_n}\equiv\{x: [0;1]\ni x=\Delta^{-\tilde Q} _{c_1c_2...c_ni_{n+1}i_{n+2}...i_{n+k}...}, i_{n+k}\in N^0 _{m_{n+k}}, k\in \mathbb N\},
$$
where  $c_1,c_2,...,c_n$ is a fixed  tuple of symbols from  $N^{0} _{m_1}, N^{0} _{m_2},..., N^{0} _{m_n}$ respectively. 
\end{definition}

Since \cite{S. Serbenyuk nega-tilde Q-representation}, the following statements are true.

\begin{lemma}
Cylinders $\Delta^{-\tilde Q} _{c_1c_2...c_n}$  have the following properties:
\begin{enumerate}
\item a cylinder  $\Delta^{-\tilde Q} _{c_1c_2...c_n}$ is the following closed interval:
$$
\Delta^{-\tilde Q} _{c_1c_2...c_n}=\left[\Delta^{-\tilde Q} _{c_1c_2...c_nm_{n+1}0m_{n+3}0m_{n+5}...};\Delta^{-\tilde Q} _{c_1c_2...c_n0m_{n+2}0m_{n+4}0m_{n+6}...}\right], \text{if $n$ is an odd},  
$$
$$
\Delta^{-\tilde Q} _{c_1c_2...c_n}=\left[\Delta^{-\tilde Q} _{c_1c_2...c_n0m_{n+2}0m_{n+4}0m_{n+6}...};\Delta^{-\tilde Q} _{c_1c_2...c_nm_{n+1}0m_{n+3}0m_{n+5}...}\right], \text{if $n$ is an even};
$$
\item for any  $n \in \mathbb N$
$$
\Delta^{-\tilde Q} _{c_1c_2...c_nc}\subset \Delta^{-\tilde Q} _{c_1c_2...c_n};
$$
\item for all $n \in \mathbb N$
$$
\Delta^{-\tilde Q} _{c_1c_2...c_n}=\bigcup^{m_{n+1}} _{c=0} {\Delta^{-\tilde Q} _{c_1c_2...c_nc}};
$$
\item
$$
\sup\Delta^{-\tilde Q} _{c_1c_2...c_{n-1}c}=\inf\Delta^{-\tilde Q} _{c_1c_2...c_{n-1}[c+1]}, \text{if $n$ is an odd}, 
$$
$$
\sup \Delta^{-\tilde Q} _{c_1c_2...c_{n-1}[c+1]}=\inf \Delta^{-\tilde Q} _{c_1c_2...c_{n-1}c}, \text{if  $n$ is an even};
$$
\item for any  $n \in \mathbb N$
$$
|\Delta^{-\tilde Q} _{c_1c_2...c_n}|=\prod^{n} _{j=1}{\tilde q_{c_j,j}};
$$
\item for arbitrary  $x \in [0;1]$
$$
\bigcap^{\infty} _{n=1}{\Delta^{-\tilde Q} _{c_1c_2...c_n}}=x=\Delta^{-\tilde Q} _{c_1c_2...c_n...}.
$$
\end{enumerate}
\end{lemma}

\begin{lemma}
\label{lem: nega-tilde Q-representation 1}
For a nega-$\tilde Q$-representation the following  identities are true: 
$$
\Delta^{-\tilde Q} _{i_1i_2...i_n...}\equiv \Delta^{\tilde Q} _{i_1[m_2-i_2]...i_{2k-1}[m_{2k}-i_{2k}]...},  ~\Delta^{\tilde Q} _{i_1i_2...i_n...}\equiv\Delta^{-\tilde Q} _{i_1[m_2-i_2]...i_{2k-1}[m_{2k}-i_{2k}]...}.
$$
\end{lemma}
That is
$$
x=\Delta^{-\tilde Q} _{i_1i_2...i_n...}\equiv a_{i_1,1}+\sum^{\infty} _{n=2}{\left[\tilde a_{i_n,n}\prod^{n-1} _{j=1}{\tilde q_{i_j,j}}\right]},
$$
where
$$
\tilde a_{i_n,n}=\begin{cases}
a_{i_n,n},&\text{if $n$ is an odd;}\\
a_{m_n-i_n,n},&\text{if $n$ is an even,}
\end{cases}
$$
$$
\tilde q_{i_n,n}=\begin{cases}
q_{i_n,n},&\text{if $n$ is an odd;}\\
q_{m_n-i_n,n},&\text{if $n$ is an even.}
\end{cases}
$$

\section{The main object of the research is an infinite system of functional equations}

Let we have matrixes of the same dimension $\tilde Q=||q_{i,n}||$ (properties of the last-mentioned matrix were considered earlier) and  $P=||p_{i,n}||$, where  $i=\overline{0,m_n}$, $m_n\in \mathbb N\cup \{0\}$, $n=1,2,...$, and for elements  $p_{i,n}$ of $P$ the following system of conditions is true: 

$$
\left\{
\begin{aligned}
\label{eq: tilde Q 1}
1^{\circ}.~~~~~~~~~~~~~~~~~~~~~~~~~~ ~~~~~ p_{i,n}\in (-1;1);\\
2^{\circ}.  ~~~~~~~~~~~~~~~~~~~~~\forall n \in \mathbb N: \sum^{m_n}_{i=0} {p_{i,n}}=1;\\
3^{\circ}.  ~~~~\forall (i_n), i_n \in \mathbb N \cup \{0\}: \prod^{\infty} _{n=1} {|p_{i_n,n}|}=0;\\
4^{\circ}.~~~~~~~~~~~~~~~\forall  i_n \in \mathbb N: 0<\sum^{i_n-1} _{i=0} {p_{i,n}}<1.\\
\end{aligned}
\right.
$$

Consider an infinite system of functional equations 

\begin{equation}
\label{def: function eq. system - nega-tilde Q 1}
f\left(\hat \varphi^k(x)\right)=\tilde \beta_{i_{k+1},k+1}+\tilde p_{i_{k+1},k+1}f\left(\hat \varphi^{k+1}(x)\right),
\end{equation}
where $k=0,1,...$, $\hat \varphi$ is a shift operator of a  $\tilde Q$-expansion, $x=\Delta^{-\tilde Q} _{i_1(x)i_2(x)...i_n(x)...}\equiv \ \equiv \Delta^{\tilde Q} _{i_1(x)[m_2-i_2(x)]i_3...i_{2k+1}(x)[m_{2k+2}-i_{2k+2}(x)]...}$,
$$
\tilde p_{i_n,n}=\begin{cases}
p_{i_n,n},&\text{if $n$ is an odd;}\\
p_{m_n-i_n,n},&\text{if $n$ is an even, }
\end{cases}
$$
$$
\beta_{i_n,n}=\begin{cases}
\sum^{i_n-1} _{i=0} {p_{i,n}}>0,&\text{if $i_n \ne 0$;}\\
0,&\text{if $i_n=0$,}
\end{cases}~~~
\tilde \beta_{i_n,n}=\begin{cases}
\beta_{i_n,n},&\text{if $n$ is an odd;}\\
\beta_{m_n-i_n,n},&\text{if $n$ is an even. }
\end{cases}
$$

Since the equality \eqref{shift operator-tilde Q}  is true,  one can to write the system \eqref{def: function eq. system - nega-tilde Q 1} by 
\begin{equation}
\label{def: function eq. system - nega-tilde Q 2}
f\left(\tilde a_{i_k,k}+\tilde q_{i_k,k}\hat \varphi^k(x)\right)=\tilde \beta_{i_{k},k}+\tilde p_{i_{k},k}f\left(\hat \varphi^{k}(x)\right),
\end{equation}
where $k=1,2,...$, $i\in N^0 _{m_k}$.

\begin{lemma} 
The function  
\begin{equation}
\label{def: function - nega-tilde Q}
F(x)=\beta_{i_1(x),1}+\sum^{\infty} _{k=2}{\left[\tilde \beta_{i_k(x),k}\prod^{k-1} _{j=1}{\tilde p_{i_j(x),j}}\right]}
\end{equation}
  is a well-defined function at an arbitrary point from  $[0;1]$. 
\end{lemma}
\begin{proof} Let  $x$ be a nega-$\tilde Q$-irrational number. The series \eqref{def: function - nega-tilde Q} is an absolutely convergent. The last-mentioned statement follows from conditions 
 $\tilde\beta_{i_k(x),k}\in [0;1)$, $|\tilde p_{i_k(x),k}|<1$ and a convergence of the series  
$$
\beta_{i_1(x),1}+\sum^{\infty} _{k=2}{\left[\tilde \beta_{i_k(x),k}\prod^{k-1} _{j=1}{|\tilde p_{i_j(x),j}|}\right]}.
$$
The convergence is proved by analogy with a proof of a convergence of the series \eqref{eq: tilde Q-representation1}, when 
$\tilde p_{i_k(x),k}\ne 0$ for all $k\in\mathbb N$.

Let   $x$ be a nega-$\tilde Q$-rational number. Consider the 	
difference  
$$
\Delta=F(\Delta^{-\tilde Q} _{i_1i_2...i_{n-1}i_nm_{n+1}0m_{n+3}0m_{n+5}...})-F(\Delta^{-\tilde Q} _{i_1i_2...i_{n-1}[i_n-1]0m_{n+2}0m_{n+4}...}).
$$

Let $n$ be an even. Then
$$
\Delta=\left(\prod^{n-1} _{j=1} {\tilde{p}_{\varepsilon_j,j}}\right)\times
$$
$$
\times \left( \beta_{m_n-i_n,n}-\beta_{m_n-i_n+1,n}+p_{m_n-i_n,n}\left (\beta_{m_{n+1},n+1}+\sum^{\infty} _{l=2}{\beta_{m_{n+l},n+l}\left[\prod^{n+l-1} _{j=n+1}{p_{m_j,j}}\right]}\right)\right)=
$$
$$
=\left(\prod^{n-1} _{j=1} {\tilde{p}_{\varepsilon_j,j}}\right)\cdot(-p_{m_n-i_n,n}+p_{m_n-i_n,n})=0.
$$

If $n$ is an odd, then
$$
\Delta=\left(\prod^{n-1} _{j=1} {\tilde{p}_{\varepsilon_j,j}}\right)\left( \beta_{i_n,n}-\beta_{i_n-1,n}-p_{i_n-1,n}\left (\beta_{m_{n+1},n+1}+\sum^{\infty} _{l=2}{\beta_{m_{n+l},n+l}\left[\prod^{n+l-1} _{j=n+1}{p_{m_j,j}}\right]}\right)\right)=0.
$$
\end{proof}
\begin{theorem}
The infinite system \eqref{def: function eq. system - nega-tilde Q 1}  of functional equations has the unique solution 
$$
f(x)=\beta_{i_1(x),1}+\sum^{\infty} _{k=2}{\left[\tilde \beta_{i_k(x),k}\prod^{k-1} _{j=1}{\tilde p_{i_j(x),j}}\right]}
$$
in the class of determined and bounded on $[0;1]$ functions.
\end{theorem}
\begin{proof} Really, for arbitrary $ x=\Delta^{-\tilde Q} _{i_1(x)i_2(x)...i_k(x)...}$ from $[0;1]$
$$
f(x)=\beta_{i_1(x),1}+p_{i_1(x),1}f\left(\hat \varphi(x)\right)=\beta_{i_1(x),1}+p_{i_1(x),1}\left(\beta_{m_2-i_2(x),2}+p_{m_2-i_2(x),2}f\left(\hat \varphi^2(x)\right)\right)=
$$
$$
=\beta_{i_1(x),1}+\beta_{m_2-i_2(x),2}p_{i_1(x),1}+p_{i_1(x),1}p_{m_2-i_2(x),2}\left(\beta_{i_3(x),3}+p_{i_3(x),3}f\left(\hat \varphi^3(x)\right)\right)=...=
$$
$$
=\beta_{i_1(x),1}+\tilde\beta_{i_2(x),2}\tilde p_{i_1(x),1}+\tilde\beta_{i_3(x),3}\tilde p_{i_1(x),1}\tilde p_{i_2(x),2}+...+\tilde\beta_{i_k(x),k}\prod^{k-1} _{j=1}{\tilde p_{i_j(x),j}}+\left(\prod^{k} _{j=1}{\tilde p_{i_j(x),j}}\right)f\left(\hat \varphi^k(x)\right).
$$
Since $\hat\varphi^k(x)\in [0;1]$ for arbitraries $x\in [0;1]$ and $k\in Z_0$ and the function  $f$ is determined at all points from  $[0;1]$ and the function is bounded on the segment $[0;1]$ (i. e, there exists $M>0$ such that for any  $x\in [0;1]$: $|f(x)|\le M$) and the  condition 
$$
\prod^{k} _{j=1}{\tilde p_{i_j(x),j}}\le\prod^{k} _{j=1}{|\tilde p_{i_j(x),j}|}\to 0 ~~~(k\to\infty)
$$
holds, it follows that
$$
f\left(\Delta^{-\tilde Q} _{i_1i_2...i_k...}\right)=\lim_{k\to\infty}{\left(\beta_{i_1(x),1}+\sum^{k} _{n=2}{\left[\tilde \beta_{i_n(x),n}\prod^{n-1} _{j=1}{\tilde p_{i_j(x),j}}\right]}+\left(\prod^{k} _{j=1}{\tilde p_{i_j(x),j}}\right)f\left(\hat \varphi^k(x)\right)\right)}=
$$
$$
=\lim_{k\to\infty}{\left(\beta_{i_1(x),1}+\sum^{k} _{n=2}{\left[\tilde \beta_{i_n(x),n}\prod^{n-1} _{j=1}{\tilde p_{i_j(x),j}}\right]}\right)}=\beta_{i_1(x),1}+\sum^{\infty} _{k=2}{\left[\tilde \beta_{i_k(x),k}\prod^{k-1} _{j=1}{\tilde p_{i_j(x),j}}\right]}.
$$
\end{proof}

So, by the system   \eqref{def: function eq. system - nega-tilde Q 1} one can to model  the class $\Lambda_F$  of determined and bounded on [0; 1] functions, where the fixed matrix $P$ determines the unique function $y= F(x)$  such that
$$
x= \sum^{i_1-1} _{i=0}{q_{i,1}}+\sum^{\infty} _{n=2}{\left[(-1)^{n-1}\tilde \delta_{i_n,n}\prod^{n-1} _{j=1}{\tilde q_{i_j,j}}\right]}+\sum^{\infty} _{n=1}{\left(\prod^{2n-1} _{j=1}{\tilde q_{i_j,j}}\right)},
$$
$$
y= F(x)=\sum^{i_1-1} _{i=0}{p_{i,1}}+\sum^{\infty} _{n=2}{\left[(-1)^{n-1}\tilde \zeta_{i_n,n}\prod^{n-1} _{j=1}{\tilde p_{i_j,j}}\right]}+\sum^{\infty} _{n=1}{\left(\prod^{2n-1} _{j=1}{\tilde p_{i_j,j}}\right)},
$$
where
$$
\tilde \zeta_{i_{n},n}=\begin{cases}
1,&\text{if $n$ is an even and  $i_{n}=m_{n}$;}\\
\sum^{m_{n}} _{i=m_{n}-i_{n}} {p_{i,n}},&\text{if $n$ is an even and  $i_{n}\ne m_{n}$;}\\
0,&\text{if $n$ is an odd and   $i_{n}=0$;}\\
\sum^{i_n-1} _{i=0}{p_{i,n}},&\text{if $n$ is an odd and  $i_{n}\ne0$.}\\
\end{cases}
$$

\begin{remark}
An argument of the function $F\in \Lambda_F$ is  determined by nega-$\tilde Q$-representation, but an expansion of the value of the function has only "formal" view of a nega-$P$-representation and is the last  only, if all elements $p_{i,n}$ of the matrix 
$P$ are positive numbers. 
\end{remark}

\section{ Continuity and monotonicity conditions  of a solution of the functional equations system 
 \eqref{def: function eq. system - nega-tilde Q 1}}

\begin{theorem}
The function  $y=F(x)$ is a continuous function on  $[0;1]$.
\end{theorem}
\begin{proof} Let $[0;1]\ni x_0$ be a some number. 
Consider the difference 
$$
F(x)-F(x_0)=\left(\prod^{n_0} _{j=1}{\tilde p_{i_j,j}}\right)\left(\tilde \beta_{i_{n_0+1}(x),n_0+1}-\tilde \beta_{i_{n_0+1}(x_0),n_0+1}\right)+
$$
$$
+\left(\prod^{n_0} _{j=1}{\tilde p_{i_j,j}}\right)\left(\sum^{\infty} _{k=n_0+2}{\left(\tilde \beta_{i_k(x),k}\prod^{k-1} _{l=n_0+1}{\tilde p_{i_l(x),l}}\right)}-\sum^{\infty} _{k=n_0+2}{\left(\tilde \beta_{i_k(x_0),k}\prod^{k-1} _{l=n_0+1}{\tilde p_{i_l(x_0),l}}\right)}\right),
$$
where $i_{n_0+1}(x)\ne i_{n_0+1}(x_0)$, $ i_j(x)=i_j(x_0)$, $j=\overline{1,n_0}$.

Let $x_0$ be a nega-$\tilde Q$-irrational point. Since $F$ is bounded and the conditions  $x\to x_0$, $n_0\to \infty$ are equivalents,  it is easy to see that 
$$
\lim_{x\to x_0}{|F(x)-F(x_0)|}=\lim_{n_0\to \infty}{\left(\prod^{n_0} _{j=1}{\left|\tilde p_{i_j,j}\right|}\right)}=0
$$
So, $\lim_{x\to x_0}{F(x)}=F(x_0)$.

Let $x_0$  be a nega-$\tilde Q$-rational number. Let us denote 
$$
x_0=x^{(1)} _0=\Delta^{-\tilde Q} _{i_1i_2...i_{n-1}[i_n-1]0m_{n+2}0m_{n+4}...}=\Delta^{-\tilde Q} _{i_1i_2...i_{n-1}i_nm_{n+1}0m_{n+3}...}=x^{(2)} _0, ~~~\text{if $n$ is an odd,}
$$
$$
x_0=x^{(1)} _0=\Delta^{-\tilde Q} _{i_1i_2...i_{n-1}i_nm_{n+1}0m_{n+3}...}=\Delta^{-\tilde Q} _{i_1i_2...i_{n-1}[i_n-1]0m_{n+2}0m_{n+4}...}=x^{(2)} _0, ~~~\text{if $n$ is an even,}
$$
and consider the limits 
$$
\lim_{x\to x_0+0}{F(x)}=\lim_{x\to x^{(2)} _0}{F(x)}, ~~~\lim_{x\to x_0-0}{F(x)}=\lim_{x\to x^{(1)} _0}{F(x)}
$$
by  considerations  for the case of a nega-$\tilde Q$-irrational number $x_0$.

Therefore,   $F$ is a continuous function on  $[0;1]$.
\end{proof}

\begin{lemma}
A value of the increment 
$$
\mu_{F} {\left(\Delta^{-\tilde Q} _{c_1c_2...c_n}\right)}=F\left(\sup{\Delta^{-\tilde Q} _{c_1c_2...c_n}}\right)-F\left(\inf{\Delta^{-\tilde Q} _{c_1c_2...c_n}}\right)
$$ 
of the function $F$ on the cylinder $\Delta^{-\tilde Q} _{c_1c_2...c_n}$  is calculated by the formula
\begin{equation}
\label{eq: function nega-tilde Q pryrist}
\mu_{F} {\left(\Delta^{-\tilde Q} _{c_1c_2...c_n}\right)}=\prod^{n} _{j=1}{\tilde p_{i_j,j}}.
\end{equation}
\end{lemma}
\begin{proof} From the definition and properties of cylinders $\Delta^{-\tilde Q} _{c_1c_2...c_n}$  it follows that 
$$
\mu_{F} {\left(\Delta^{-\tilde Q} _{c_1c_2...c_n}\right)}=\begin{cases}
F\left(\Delta^{-\tilde Q} _{c_1c_2...c_n0m_{n+2}0m_{n+4}...}\right)-F\left(\Delta^{-\tilde Q} _{c_1c_2...c_nm_{n+1}0m_{n+3}...}\right),&\text{if $n$ is an odd;}\\
F\left(\Delta^{-\tilde Q} _{c_1c_2...c_nm_{n+1}0m_{n+3}...}\right)-F\left(\Delta^{-\tilde Q} _{c_1c_2...c_n0m_{n+2}0m_{n+4}...}\right),&\text{if $n$ is an even.}
\end{cases}
$$
Therefore,
$$
\mu_{F} {\left(\Delta^{-\tilde Q} _{c_1c_2...c_n}\right)}=\left(\beta_{m_{n+1},n+1}+\beta_{m_{n+2},n+2}p_{m_{n+1},n+1}+\beta_{m_{n+3},n+3}p_{m_{n+1},n+1}p_{m_{n+2},n+2}+...\right)\prod^{n} _{j=1}{\tilde p_{i_j,j}},
$$
The equality  \eqref{eq: function nega-tilde Q pryrist} follows from the last-mentioned expression.
\end{proof}

\begin{theorem}
\label{th: nega-tilde Q function property1}
The function  $y=F(x)$ is:
\begin{itemize}

\item a monotonic non-decreasing function, if elements $p_{i,n}$ of the matrix $P$ are non-negative, and a strictly increasing function, if all elemets of the matrix $P$ are positive numbers;

\item a non-monotonic function, that has at least one monotonicity interval on $[0;1]$, if the matrix $P$  does not have zeros
and there exists only the finite tuple $\{p_{i,n}\}$ of the elements $p_{i,n}<0$ of the matrix $P$;

\item a function, that does not have  monotonicity intervals on $[0;1]$, if the matrix $P$  does not have zeros
and there exists the infinite subsequence $(n_k)$ of sequence of positive integer numbers such that for arbitrary $k \in \mathbb N$ there exists at least one element  $p_{i,n_k}<0$ of $P$, where $i=\overline{0,m_{n_k}}$;

\item a constant almost everywhere on $ [0, 1]$ function, if  there exists the infinite subsequence $(n_k)$ of a sequence of positive integer numbers such that for arbitrary $k \in \mathbb N$ there exists at least one element  $p_{i,n_k}=0$ of the matrix  $P$, where $i=\overline{0,m_{n_k}}$.
\end{itemize}
\end{theorem}
\begin{proof}
\emph{The first} and \emph{the second statements} follow from \eqref{eq: function nega-tilde Q pryrist}.

\emph{The third statement}. Let us choose a some number $n_0$ such that the number $n_0+1$  belongs to a subsequence $(n_k)$ of a positive integers sequence and the tuple $c_1,c_2,...,c_{n_0}$, that  the condition  $\mu_{F} {\left(\Delta^{-\tilde Q} _{c_1c_2...c_{n_0}}\right)}>0$ holds.  It is easy to see that there exist nega-$\tilde Q$-symbols $i_{n_0+1}$ and $i^{'} _{n_0+1}$ such that  $\tilde p_{i_{n_0+1},n_0}<0$, $\tilde p_{i^{'} _{n_0+1},n_0}>0$. In the case for
$$
\Delta^{-\tilde Q} _{c_1c_2...c_{n_0}}\supset \left(\Delta^{-\tilde Q} _{c_1c_2...c_{n_0}i_{n_0+1}}\cup \Delta^{-\tilde Q} _{c_1c_2...c_{n_0}i^{'} _{c_{n_0}+1}}\right),
$$
since the equation \eqref{eq: function nega-tilde Q pryrist}, it follows that 
$$
\mu_{F} {\left(\Delta^{-\tilde Q} _{c_1c_2...c_{n_0}i_{n_0+1}}\right)}<0<\mu_{F} {\left(\Delta^{-\tilde Q} _{c_1c_2...c_{n_0}i^{'} _{n_0+1}}\right)}.
$$

Since the function $F$ is a monotonically increasing  function on a some segment,  the function does not have  intervals of increasing and decreasing on the segment simultaneously. Therefore, the last-mentioned double inequality is a contradiction. So, the function $F$ does not have  monotonicity intervals on $[0;1]$.

Since the equality \eqref{eq: function nega-tilde Q pryrist} holds and the value of the  Lebesgue measure of the set 
$$
C\left[-\tilde Q,\overline{(V_{n_k})}\right]\equiv\left\{x: x=\Delta^{-\tilde Q} _{i_1i_2...i_n...}, i_{n_k}\notin V_{n_k}\right\},
$$
where  $(n_k)$ is a fixed subsequence of a positive integers sequence and   $(V_{n_k})$ is a sequence of the subsets  $V_{n_k}$  of the sets $N^0 _{m_{n_k}}$ of nega-$\tilde Q$-symbols such that
$$
V_{n_k}\equiv\{i: i\in N^0 _{m_{n_k}}, p_{i,n_k}=0\},
$$
equals to zero, \emph{the fourth statement} is true.

\begin{proposition}
\label{pr: set1}
Let we have the set 
$$
C\left[\tilde Q, (N^0 _{m_n}\setminus{\{{i^{*} _n}\})}\right]\equiv\left\{x: x=\Delta^{-\tilde Q} _{i_1i_2...i_n...}, i_n\in N^0 _{m_n}\setminus{\{{i^{*} _n}\}}\right\},
$$
 where $(i^{*} _n)$ is a fixed sequence of  $\tilde Q$-symbols, that $i^{*} _1\in N^0 _{m_1}, i^{*} _2\in N^0 _{m_2},...$.

The following equality
$$
\lambda\left(C\left[\tilde Q, (N^0 _{m_n}\setminus{\{{i^{*} _n}\})}\right]\right)=0
$$
 holds, where $\lambda(\cdot)$ is the Lebesgue measure of a set.
\end{proposition}
\begin{proof}

Let us denote by $V_{0,n}=N^0 _{m_n}\setminus{\{{i^{*} _n}\}}$,
$$
E_n=\bigcup_{{i_1\ne i^{*} _1,..., i_n\ne i^{*} _n}}{\Delta^{-\tilde Q} _{i_1i_2...i_n}},
$$	
at that, $E_n=E_{n+1}\cup\overline{E_{n+1}}$, $C[\tilde Q, (V_{0,n})]=\cap^{\infty} _{n=1}{E_n}$, where
$E_{n+1}\subset E_n$ òà $E_0=[0;1]$.

From the property of continuity of the Lebesgue measure, we obtain that
$$
\lambda\left(C[\tilde Q, (V_{0,n})]\right)=\lim_{n\to\infty}{\lambda(E_n)}=\lim_{n\to\infty}{\prod^{n} _{k=1}{(1-q_{i^{*} _k,k})}}=0,
$$
because
$$
\lambda(E_{n+1})=\sum_{\substack{i_1\ne i^{*} _1,\\
...,\\
i_{n+1}\ne i^{*} _{n+1}}}{\left(q_{i_1,1}...q_{i_{n+1},n+1}\right)}=\lambda(E_{n})- q_{i^{*} _{n+1},n+1}\sum_{\substack{i_1\ne i^{*} _1, \\
...,\\
i_{n}\ne i^{*} _{n}}}{\left(q_{i_1,1}...q_{i_{n},n}\right)}
$$
and 
$$
1=\sum_{i_1\in N^0 _{m_1},...,i_n \in N^0 _{m_n}}{\left(q_{i_1,1}q_{i_2,2}...q_{i_n,n}\right)}=
$$
$$
=\sum_{\substack{i_1\ne i^{*} _1,\\
...,\\
i_{n}\ne i^{*} _{n}}}{\left(q_{i_1,1}...q_{i_{n},n}\right)}+ q_{i^{*} _{n},n}\sum_{\substack{i_1\ne i^{*} _1, \\
...,\\
i_{n-1}\ne i^{*} _{n-1}}}{\left(q_{i_1,1}...q_{i_{n-1},n-1}\right)}+...+q_{i^{*} _1,1}q_{i^{*} _2,2}...q_{i^{*} _n,n}.
$$
So, the proposition \ref{pr: set1} is true. \end{proof}

Since the relation between $\tilde Q$- and nega-$\tilde Q$-representation is given, proofs of equalities 
$\lambda\left(C\left[-\tilde Q,\overline{(V_{n_k})}\right]\right)=0$, $\lambda\left(C\left[\tilde Q, (N^0 _{m_n}\setminus{\{{i^{*} _n}\})}\right]\right)=0$ are analogy. 
\end{proof}

\begin{corollary}
The function  $F$ is a bijective mapping on   $[0;1]$, if all elements of the matrix  $P$ are positive numbers. 
\end{corollary}

\section{Integral properties}

\begin{theorem}
The Lebesgue integral of the function $F$ can be calculated by the
following formula
$$
\int^1 _0{F(x)dx}=z_1+\sum^{\infty} _{n=2}{\left(z_n\prod^{n-1} _{k=1}{\sigma_k}\right)},~~~\mbox{where}
$$
$$
z_n= \tilde\beta_{0,n}\tilde q_{0,n}+ \tilde\beta_{1,n}\tilde q_{1,n}+...+ \tilde\beta_{m_n,n}\tilde q_{m_n,n}=\beta_{0,n} q_{0,n}+ \beta_{1,n} q_{1,n}+...+ \beta_{m_n,n} q_{m_n,n},
$$
$$
\sigma_n= \tilde p_{0,n}\tilde q_{0,n}+ \tilde p_{1,n}\tilde q_{1,n}+...+ \tilde p_{m_n,n}\tilde q_{m_n,n}= p_{0,n}q_{0,n}+  p_{1,n} q_{1,n}+...+  p_{m_n,n} q_{m_n,n}.
$$
\end{theorem}
\begin{proof} From the equality 
\eqref{shift operator-tilde Q}, it follows that 
$$
d(\hat\varphi^n(x))=\tilde q_{i_{n+1},n+1}d(\hat\varphi^{n+1}(x)).
$$
By the definition of the function and the additive property of the  Lebesgue integral we  obtain that
$$
\int^1 _0 {{F}(x)dx}=\int^{a_{1,1}} _0 {{F}(x)dx}+\int^{a_{2,1}} _{a_{1,1}} {{F}(x)dx}+...+\int^{1} _{a_{m_1,1}} {{F}(x)dx}=
$$
$$
=\int^{a_{1,1}} _0 {p_{0,1}{F}(\hat{\varphi}(x))dx}+\int^{a_{2,1}} _{a_{1,1}} {\left[\beta_{1,1}+p_{1,1}{F}(\hat{\varphi}(x))\right]dx}+...+\int^{1} _{a_{m_1,1}} {\left[\beta_{m_1,1}+p_{m_1,1}{F}(\hat{\varphi}(x))\right]dx}=
$$
$$
=p_{0,1}q_{0,1}\int^{1} _{0} {{F}(\hat{\varphi}(x))d(\hat{\varphi}(x))}+\beta_{1,1}q_{1,1}+p_{1,1}q_{1,1}\int^{1} _{0} {{F}(\hat{\varphi}(x))d(\hat{\varphi}(x))}+...+\beta_{m_1,1}q_{m_1,1}+
$$
$$
+p_{m_1,1}q_{m_1,1}\int^{1} _{0} {{F}(\hat{\varphi}(x))d(\hat{\varphi}(x))}=(\beta_{0,1}q_{0,1}+\beta_{1,1}q_{1,1}+...+\beta_{m_1,1}q_{m_1,1})+
$$
$$
+(p_{0,1}q_{0,1}+p_{1,1}q_{1,1}+...+p_{m_1,1}q_{m_1,1})\int^{1} _{0} {{F}(\hat{\varphi}(x))d(\hat{\varphi}(x))}=z_1+\sigma_1\int^{1} _{0} {{F}(\hat{\varphi}(x))d(\hat{\varphi}(x))}=z_1+
$$
$$
+\sigma_1\left(\int^{a_{1,2}} _{0}{(\tilde \beta_{m_2,2} +\tilde p_{m_2,2}{F}(\hat{\varphi}^2(x)))d(\hat{\varphi}(x))}+...+\int^{1} _{a_{m_2,2}}{(\tilde \beta_{0,2} +\tilde p_{0,2}{F}(\hat{\varphi}^2(x)))d(\hat{\varphi}(x))}\right)=
$$
$$
=z_1+\sigma_1(\beta_{0,2}q_{0,2}+\beta_{1,2}q_{1,2}+...+\beta_{m_2,2}q_{m_2,2}+(p_{0,2}q_{0,2}+...+ 
$$
$$
 +p_{m_2,2}q_{m_2,2})\int^{1} _{0} {{F}(\hat{\varphi^2}(x))d(\hat{\varphi^2}(x))})=
z_1+z_2\sigma_1+\sigma_1\sigma_2\int^{1} _{0} {{F}(\hat{\varphi^2}(x))d(\hat{\varphi^2}(x))}=...=
$$
$$
=z_1+z_2\sigma_1+z_3\sigma_1\sigma_2+...+z_n\sigma_1\sigma_2...\sigma_{n-1}+\sigma_1\sigma_2...\sigma_n\int^{1} _{0} {{F}(\hat{\varphi^n}(x))d(\hat{\varphi^n}(x))}=....
$$

Continuing the process indefinitely we obtain that
$$
\int^1 _0{F(x)dx}=z_1+\sum^{\infty} _{n=2}{\left(z_n\prod^{n-1} _{k=1}{\sigma_k}\right)}.
$$
\end{proof}

\section{Modeling of the singular distribution functions}

\begin{lemma}
If the function  $F$ has a derivative   $F^{'}(x_0)$ at a nega-$\tilde Q$-irrational point  $x_0$, then    
$$
F^{'}(x_0)=\lim_{n\to\infty}{\left(\prod^{n} _{j=1}{\frac{\tilde p_{i_j(x_0),j}}{\tilde q_{i_j(x_0),j}}}\right)}=\prod^{\infty} _{n=1}{\frac{\tilde p_{i_n(x_0),n}}{\tilde q_{i_n(x_0),n}}}.
$$
\end{lemma}
\begin{proof}
The statement is true, because the Property~6 of cylinders $\Delta^{-\tilde Q} _{c_1c_2...c_n}$ is true and the conditions    $ x\to x_0, n \to \infty$ are equivalent.
\end{proof}

\begin{theorem}
If for all  $n\in\mathbb N$ and  $i=\overline{0,m_n}$  it is true that   $p_{i,n}\ge0$, then the unique solution of the functional equations system   \eqref{def: function eq. system - nega-tilde Q 1}  is a continuous function of probabilities distribution on  $[0;1]$.
\end{theorem}
\begin{proof} It is easy to see that 
$$
F(0)=F\left(\Delta^{-\tilde Q} _{0m_20m_4...0m_{2k}...}\right)=\beta_{0,1}+\sum^{\infty} _{n=2}{\left[\beta_{0,n}\prod^{n-1} _{j=1}{p_{0,j}}\right]}=0,
$$
$$
F(1)=F\left(\Delta^{-\tilde Q} _{m_10m_3...0m_{2k-1}...}\right)=\beta_{m_1,1}+\sum^{\infty} _{n=2}{\left[\beta_{m_n,n}\prod^{n-1} _{j=1}{p_{m_j,j}}\right]}=1.
$$

Let  $x_1=\Delta^{-\tilde Q} _{i_1(x_1)i_2(x_1)...i_{n}(x_1)...}$, $x_2=~\Delta^{-\tilde Q} _{i_1(x_2)i_2(x_2)...i_{n}(x_2)...}$ be the some numbers from  $[0;1]$ such that  $x_1<x_2$. Therefore, there exists a number $n_0$, that $i_j(x_1)=i_j(x_2)$ for all  $j=\overline{1,n_0-1}$ and $i_{n_0}(x_1)<i_{n_0}(x_2)$, if $n_0$ is an odd, or  $i_{n_0}(x_1)>i_{n_0}(x_2)$, if   $n_0$ is an even.

Whence, 
$$
F(x_2)-F(x_1)=\left(\prod^{n_0-1} _{j=1} {\tilde{p}_{i_{j}(x_2),j}}\right)\cdot(\tilde{\beta}_{i_{n_0}(x_2),n_0}-\tilde{\beta}_{i_{n_0}(x_1),n_0}+
$$
$$
+\sum^{\infty} _{k=1} {\left(\tilde{\beta}_{i_{n_0+k}(x_2),n_0+k}\prod^{k-1} _{j=0} {\tilde{p}_{i_{n_0+j}(x_2),n_0+j}}\right)}-\sum^{\infty} _{k=1} {\left(\tilde{\beta}_{i_{n_0+k}(x_1),n_0+k}\prod^{k-1} _{j=0} {\tilde{p}_{i_{n_0+j}(x_1),n_0+j}}\right)})\ge
$$
$$
\ge\left(\prod^{n_0-1} _{j=1} {\tilde{p}_{i_{j}(x_2),j}}\right)\cdot\left(\tilde{\beta}_{i_{n_0}(x_2),n_0}-\tilde{\beta}_{i_{n_0}(x_1),n_0}-\sum^{\infty} _{k=1} {\left(\tilde{\beta}_{i_{n_0+k}(x_1),n_0+k}\prod^{k-1} _{j=0} {\tilde{p}_{i_{n_0+j}(x_1),n_0+j}}\right)}\right)=\kappa,
$$
where for an odd  $n_0$
$$
\kappa\ge \left(\prod^{n_0-1} _{j=1} {\tilde{p}_{i_{j}(x_2),j}}\right)\left(p_{i_{n_0}(x_1),n_0}+p_{i_{n_0}(x_1)+1,n_0}+...+p_{i_{n_0}(x_2)-1,n_0}-p_{i_{n_0}(x_1),n_0}\max_{x\in[0;1]}{F\left(\hat\varphi^{n_0}(x_1)\right)}\right)=
$$
$$
=\left(\prod^{n_0-1} _{j=1} {\tilde{p}_{i_{j}(x_2),j}}\right)(p_{i_{n_0}(x_1)+1,n_0}+...+p_{i_{n_0}(x_2)-1,n_0})\ge 0.
$$

Let  $n_0$ be an even number. Then
$$
\kappa\ge \left(\prod^{n_0-1} _{j=1} {\tilde{p}_{i_{j}(x_2),j}}\right)(p_{m_{n_0}-i_{n_0}(x_1),n_0}+p_{m_{n_0}-i_{n_0}(x_1)+1,n_0}+...+p_{m_{n_0}-i_{n_0}(x_2)-1,n_0}-
$$
$$
-p_{m_{n_0}-i_{n_0}(x_1),n_0}\max_{x\in[0;1]}{F\left(\hat\varphi^{n_0}(x_1)\right)})=
$$
$$
=\left(\prod^{n_0-1} _{j=1} {\tilde{p}_{i_{j}(x_2),j}}\right)(p_{m_{n_0}-i_{n_0}(x_1)+1,n_0}+...+p_{m_{n_0}-i_{n_0}(x_2)-1,n_0})\ge 0.
$$

 If all elements $p_{i,n}$  of the matrix $P$ are positive numbers, then the inequality   $F(x_2)-F(x_1)>0$ holds.

Since the function $F$ is a continuous function at all points from $[0;1]$, it follows that the function is a continuous function of probabilities distribution on  $[0;1]$.
\end{proof}

Let $\eta$ be a random variable, that defined by the following form 
$$
\eta=  \Delta^{\tilde Q} _{\xi_1\xi_2...\xi_{n}...},
$$
where
$$
\xi_n=\begin{cases}
i_n,&\text{if $n$ is an odd;}\\
m_n-i_n,&\text{if  $n$ is an even,}
\end{cases}
$$
$n=1,2,3,...$,  the digits $\xi_n$ are  random and taking the values $0,1,...,m_n$ with probabilities ${p}_{0,n}, {p}_{1,n}, ..., {p}_{m_n,n}$. That is  $\xi_n$ are independent and $P\{\xi_n=i_n\}=p_{i_n,n}$, $i_n \in N^0 _{m_n}$. 

\begin{theorem}
The distribution function $\tilde{F}_{\eta}$ of the random variable $\eta$ can be
represented by
$$
\tilde{F}_{\eta}(x)=\begin{cases}
0,&\text{ $x< 0$;}\\
\beta_{i_1(x), 1}+\sum^{\infty} _{n=2} {\left[\tilde{\beta}_{i_n(x),n} \prod^{n-1} _{j=1} {\tilde{p}_{i_j(x),j}}\right]},&\text{ $0 \le x<1$;}\\
1,&\text{ $x\ge 1$.}
\end{cases}
$$
\end{theorem}
\begin{proof} Let $k\in\mathbb N$. The statement follows from the equalities 
$$
\{\eta<x\}=\{\xi_1<i_1(x)\}\cup\{\xi_1=i_1(x),\xi_2<m_2-i_2(x)\}\cup...\cup
$$
$$
\cup\{\xi_1=i_1(x),\xi_2=m_2-i_2(x),...,\xi_{2k-1}<i_{2k-1}(x)\}\cup
$$
$$
\cup\{\xi_1=i_1(x),\xi_2=m_2-i_2(x),...,\xi_{2k-1}=i_{2k-1}(x),\xi_{2k}<m_{2k}-i_{2k}(x)\}\cup...,
$$
$$
P\{\xi_1=i_1(x),\xi_2=m_2-i_2(x),...,\xi_{2k-1}<i_{2k-1}(x)\}=\beta_{\varepsilon_{2k-1}(x),2k-1}\prod^{2k-2} _{j=1} {\tilde{p}_{i_{j}(x),j}},
$$
$$
P\{\xi_1=i_1(x),\xi_2=m_2-\varepsilon_2(x),...,\xi_{2k}<m_{2k}-\varepsilon_{2k}(x)\}=\beta_{m_{2k}-\varepsilon_{2k}(x),2k}\prod^{2k-1} _{j=1} {\tilde{p}_{i_{j}(x),j}}
$$
and the definition  of a  distribution function. 
\end{proof}

One can to formulate the following conclusions by the statements  \cite[p.~170]{Pra98} .

\begin{lemma}
Let for any  $n\in\mathbb N$ and $i=\overline{0,m_n}$ the inequality   $p_{i,n}\ge 0$ holds.

The function  $F$ is a singular function of Cantor type iff 
\begin{enumerate}
\item
$$
\prod^{\infty} _{n=1}{\left(\sum_{i: \tilde p_{i,n}>0}{\tilde q_{i,n}}\right)}=0
$$
or
\item
$$
\sum^{\infty} _{n=1}{\left(\sum_{i: \tilde p_{i,n}=0}{\tilde q_{i,n}}\right)}=\infty.
$$
\end{enumerate}
\end{lemma}

\section{Modeling of functions, that does not have a derivative at any a nega-$\tilde Q$-rational point}

\begin{theorem}
If the following properties of the matrix  $P$ hold:
\begin{itemize}
\item for all $n \in \mathbb N$, $i_n \in N^1 _{m_n}\equiv\{1,2,...,m_n\}$
$$
p_{i_n,n}\cdot p_{i_n-1,n}<0;
$$
\item the conditions 
$$
\lim_{n \to \infty} {\prod^{n} _{k=1} {\frac{p_{0,k}}{q_{0,k}}}}\ne  0, \lim_{n \to \infty} {\prod^{n} _{k=1} {\frac{p_{m_k,k}}{q_{m_k,k}} }}\ne 0
$$
\end{itemize}
hold simultaneously, then the unique solution of the functional equations system    \eqref{def: function eq. system - nega-tilde Q 1}  does not have a finite or an infinite derivative at any nega-$\tilde Q$-rational point from the segment $[0;1]$.
\end{theorem}

\begin{proof} Let  $x_0$ be a some  nega-$\tilde Q$-rational point. That is 
$$
x_0=\Delta^{-\tilde Q} _{i_1i_2...i_{n-1}i_nm_{n+1}0m_{n+3}0m_{n+5}...}=\Delta^{-\tilde Q} _{i_1i_2...i_{n-1}[i_n-1]0m_{n+2}0m_{n+4}...}, 
$$
$i_n \ne 0$.

In a case of an odd  $n$  let us denote 
$$
x_0=x^{(1)}_0=\Delta^{-\tilde Q} _{i_1i_2...i_{n-1}i_nm_{n+1}0m_{n+3}0m_{n+5}...}=\Delta^{-\tilde Q} _{i_1i_2...i_{n-1}[i_n-1]0m_{n+2}0m_{n+4}-...}=x^{(2)}_0.
$$
If $n$ is an even, then let us denote 
$$
x_0=x^{(1)}_0=\Delta^{-\tilde Q} _{i_1i_2...i_{n-1}[i_n-1]0m_{n+2}0m_{n+4}...}=\Delta^{-\tilde Q} _{i_1i_2...i_{n-1}i_nm_{n+1}0m_{n+3}0m_{n+5}...}=x^{(2)}_0.
$$
Consider the sequences $(x^{'} _k)$, $(x^{''} _k)$ such that $x^{'} _k\to x_0$, $x^{''} _k\to x_0$ as  $k\to \infty$ and for an odd number    $n$
$$
x^{'} _k=\begin{cases}
\Delta^{-\tilde Q} _{i_1...i_{n-1}i_nm_{n+1}0m_{n+3}0...m_{n+k-1}1m_{n+k+1}0m_{n+k+3}...},&\text{if $k$ is an even;}\\
\Delta^{-\tilde Q} _{i_1...i_{n-1}i_nm_{n+1}0...m_{n+k-2}0[m_{n+k}-1]0m_{n+k+2}0m_{n+k+4}...},&\text{if $k$ is an odd,}
\end{cases}
$$
$$
x^{''} _k=\begin{cases}
\Delta^{-\tilde Q} _{i_1...i_{n-1}[i_n-1]0m_{n+2}0...m_{n+k-1}00m_{n+k+2}0m_{n+k+4}...},&\text{if $k$ is an odd;}\\
\Delta^{-\tilde Q} _{i_1...i_{n-1}[i_n-1]0m_{n+2}0...m_{n+k}m_{n+k+1}0m_{n+k+3}0m_{n+k+5}...},&\text{if $k$ is an even,}
\end{cases}
$$
and for the case of an even  $n$
$$
x^{'} _k=\begin{cases}
\Delta^{-\tilde Q} _{i_1...i_{n-1}[i_n-1]0m_{n+2}0m_{n+4}0...m_{n+k-1}1m_{n+k+1}0m_{n+k+3}...},&\text{if $k$ is an odd;}\\
\Delta^{-\tilde Q} _{i_1...i_{n-1}[i_n-1]0m_{n+2}0...m_{n+k-2}0[m_{n+k}-1]0m_{n+k+2}0m_{n+k+4}...},&\text{if $k$ is an even,}
\end{cases}
$$
$$
x^{''} _k=\begin{cases}
\Delta^{-\tilde Q} _{i_1...i_{n-1}i_nm_{n+1}0m_{n+3}...0m_{n+k}m_{n+k+1}0m_{n+k+3}...},&\text{if $k$ is an odd;}\\
\Delta^{-\tilde Q} _{i_1...i_{n-1}i_nm_{n+1}0...m_{n+k-1}00m_{n+k+2}0m_{n+k+4}0m_{n+k+6}...},&\text{if $k$ is an even.}
\end{cases}
$$

Therefore, if $n$ is an odd, then 
$$
x^{'} _k-x^{(1)} _0= \Delta^{\tilde Q} _{i_1[m_2-i_2]i_3[m_4-i_4]...i_n\underbrace{0...0}_{k-1}1(0)}-\Delta^{\tilde Q} _{i_1[m_2-i_2]...i_n(0)}\equiv a_{1,n+k}\left(\prod^{n} _{j=1}{\tilde q_{i_j,j}}\right)\left(\prod^{n+k-1} _{t=n+1}{q_{0,t}}\right),
$$
$$
F\left(x^{'} _k\right)-F\left(x^{(1)} _0\right)=\beta_{1,n+k}\left(\prod^{n} _{j=1}{\tilde p_{i_j,j}}\right)\left(\prod^{n+k-1} _{t=n+1}{p_{0,t}}\right)=\left(\prod^{n} _{j=1}{\tilde p_{i_j,j}}\right)\left(\prod^{n+k} _{t=n+1}{p_{0,t}}\right),
$$
$$
x^{(2)} _0-x^{''} _k=
$$
$$
=\Delta^{\tilde Q} _{i_1[m_2-i_2]...[m_{n-1}-i_{n-1}][i_n-1]m_{n+1}m_{n+2}...}-\Delta^{\tilde Q} _{i_1[m_2-i_2]...[m_{n-1}-i_{n-1}][i_n-1]m_{n+1}m_{n+2}...m_{n+k}(0)}\equiv
$$
$$
\equiv q_{i_n-1,n}\left(\prod^{n-1} _{j=1}{\tilde q_{i_j,j}}\right)\left(\prod^{n+k} _{t=n+1}{q_{m_t,t}}\right),
$$
$$
F\left(x^{(2)} _0\right)-F\left(x^{''} _k\right)=p_{i_n-1,n}\left(\prod^{n-1} _{j=1}{\tilde p_{i_j,j}}\right)\left(\prod^{n+k} _{t=n+1}{p_{m_t,t}}\right).
$$

If $n$ is an even, then  
$$
x^{'} _k-x^{(1)} _0= \Delta^{\tilde Q} _{i_1[m_2-i_2]...i_{n-1}[m_n-i_n+1]\underbrace{0...0}_{k-1}1(0)}-\Delta^{\tilde Q} _{i_1[m_2-i_2]...i_{n-1}[m_n-i_n+1](0)}\equiv 
$$
$$
\equiv a_{1,n+k}\left(\prod^{n-1} _{j=1}{\tilde q_{i_j,j}}\right)\left(\prod^{n+k-1} _{t=n+1}{q_{0,t}}\right)q_{m_n-i_n+1,n}=\left(\prod^{n-1} _{j=1}{\tilde q_{i_j,j}}\right)\left(\prod^{n+k} _{t=n+1}{q_{0,t}}\right)q_{m_n-i_n+1,n},
$$
$$
F\left(x^{'} _k\right)-F\left(x^{(1)} _0\right)=\beta_{1,n+k}\left(\prod^{n-1} _{j=1}{\tilde p_{i_j,j}}\right)\left(\prod^{n+k-1} _{t=n+1}{p_{0,t}}\right)p_{m_n-i_n+1,n}=
$$
$$
=\left(\prod^{n-1} _{j=1}{\tilde p_{i_j,j}}\right)\left(\prod^{n+k} _{t=n+1}{p_{0,t}}\right)p_{m_n-i_n+1,n},
$$
$$
x^{(2)} _0-x^{''} _k=\Delta^{\tilde Q} _{i_1[m_2-i_2]...i_{n-1}[m_{n}-i_{n}]m_{n+1}m_{n+2}...}-\Delta^{\tilde Q} _{i_1[m_2-i_2]...i_{n-1}[m_{n}-i_{n}]m_{n+1}m_{n+2}...m_{n+k}(0)}\equiv
$$
$$
\equiv \left(\prod^{n} _{j=1}{\tilde q_{i_j,j}}\right)\left(\prod^{n+k} _{t=n+1}{q_{m_t,t}}\right),
$$
$$
F\left(x^{(2)} _0\right)-F\left(x^{''} _k\right)=\left(\prod^{n} _{j=1}{\tilde p_{i_j,j}}\right)\left(\prod^{n+k} _{t=n+1}{p_{m_t,t}}\right).
$$

So,
$$
B^{'} _k=\frac{F(x^{'} _k)-F(x_0)}{x^{'} _k-x_0}=\begin{cases}
\frac{p_{i_n,n}}{q_{i_n,n}}\left(\prod^{n-1} _{j=1} {\frac{\tilde p_{i_j,j}}{\tilde q_{i_j,j}}}\right) \left(\prod^{n+k} _{t=n+1} {\frac{p_{0,t}}{q_{0,t}}}\right),\text{ $n$ is an odd;}\\
\\
\frac{p_{m_n-i_n+1,n}}{q_{m_n-i_n+1,n}}\left(\prod^{n-1} _{j=1} {\frac{\tilde p_{i_j,j}}{\tilde q_{i_j,j}}}\right) \left(\prod^{n+k} _{t=n+1} {\frac{p_{0,t}}{q_{0,t}}}\right),\text{ $n$ is an even.}
\end{cases}
$$
$$
B^{''} _k=\frac{F(x_0)-F(x^{''} _k)}{x_0-x^{''} _k}=\begin{cases}
\frac{p_{i_n-1,n}}{q_{i_n-1,n}} \left(\prod^{n-1} _{j=1} {\frac{\tilde p_{i_j,j}}{\tilde q_{i_j,j}}}\right) \left(\prod^{n+k} _{t=n+1} {\frac{ p_{m_t,t}}{ q_{m_t,t}}}\right),\text{ $n$ is an odd;}\\
\\
 \frac{p_{m_n-i_n,n}}{q_{m_n-i_n,n}}\left(\prod^{n-1} _{j=1} {\frac{\tilde p_{i_j,j}}{\tilde q_{i_j,j}}}\right) \left(\prod^{n+k} _{t=n+1} {\frac{  p_{m_t,t}}{  q_{m_t,t}}}\right),\text{$n$ is an even.}
\end{cases}
$$

 Let us denote by 
$$
b_{0,k}=\prod^{n+k} _{t=n+1} {\frac{p_{0,t}}{q_{0,t}}}, ~~~b_{m_k,k}=\prod^{n+k} _{m=n+1} {\frac{p_{m_t,t}}{q_{m_t,t}} }.
$$

Since the conditions $\prod^{n-1} _{j=1} {(\tilde p_{i_j,j}/\tilde q_{i_j,j})}=const$, $p_{i_n,n}p_{i_n-1,n}<0$,  $p_{m_n-i_n+1,n}p_{m_n-i_n,n}<~0$  hold and the sequences $(b_{0,k})$,  $(b_{m_k,k})$  do not converge to $0$ simultaneously, for all cases the inequality  
$$
\lim_{k \to \infty}{B^{'} _k} \ne \lim_{k \to \infty}{B^{''} _k}
$$
holds. Therefore, $F$ does not have a finite or an infinite derivative at an arbitrary nega-$\tilde Q$-rational point from the segment $[0;1]$.
\end{proof}


\section{A graph of the continuous unique solution of the system  \eqref{def: function eq. system - nega-tilde Q 1}}

\begin{theorem}
Let the elements  $p_{i,n}$ of the matrix $P$ does not equal to  $0$.

Let  $x=\Delta^{-\tilde Q}_{i_1(x)i_2(x)...i_n(x)...}$ be a fixed number and the sequence  $\left(\psi_{i_n(x),n}\right)$ be a corresponding to him sequence of affine transformations of the space  $\mathbb R^2$:
$$
\psi_{i_n(x),n}:
\left\{
\begin{aligned}
x^{'}&=\tilde a_{i_n(x),n}+\tilde q_{i_n(x),n}x;\\
y^{'}& = \tilde{\beta}_{i_n(x),n}+\tilde{p}_{i_n(x),n}y,\\
\end{aligned}
\right.
$$
where $i_n \in N^{0} _{m_n}$.

Then the graph  $\Gamma_{F}$ of the function $F$ is the following set in the space $\mathbb R^2$:
$$
\Gamma_{F}=\bigcup_{x \in [0;1]}{( ... \circ \psi_{i_n(x),n}\circ ...\circ \psi_{i_2(x),2} \circ \psi_{i_1(x),1}(\Gamma_F))}.
$$
\end{theorem}
\begin{proof} Since the function $F$ is the continuous unique solution of the system \eqref{def: function eq. system - nega-tilde Q 2}, it is clear that 
$$
\psi_{i_1,1}:
\left\{
\begin{aligned}
x^{'}&= q_{i_1,1}x+ a_{i_1,1};\\
y^{'}& = {\beta}_{i_1,1}+{p}_{i_1,1}y,\\
\end{aligned}
\right.
$$
$$
\psi_{i_2,2}:
\left\{
\begin{aligned}
x^{'}&=q_{m_2-i_2,2}x+a_{m_2-i_2,2};\\
y^{'}& = {\beta}_{m_2-i_2,2}+{p}_{m_2-i_2,2}y,\\
\end{aligned}
\right.
$$
etc., whence,  
$$
\psi_{i_n,n}:
\left\{
\begin{aligned}
x^{'}&=\tilde q_{i_n,n}x+\tilde a_{i_n,n};\\
y^{'}& = \tilde{\beta}_{i_n,n}+\tilde{p}_{i_n,n}y.\\
\end{aligned}
\right.
$$
Therefore,
$$
\bigcup_{x \in [0;1]}{( ... \circ \psi_{i_n(x),n}\circ ...\circ \psi_{i_2(x),2} \circ \psi_{i_1(x),1}(\Gamma_F))}\equiv G \subset  \Gamma_{\tilde{F}},
$$
because
$$
F(x^{'})=F(\tilde a_{i_n,n}+\tilde q_{i_n,n}x)=\tilde \beta_{i_n,n}+\tilde p_{i_n,n}y=y^{'}.
$$

Let $T(x_0,F(x_0))\in \Gamma_{\tilde{F}}$, $x_0=\Delta^{-\tilde Q} _{i_1i_2...i_n...}$ be a some fixed point from  $[0;1]$. Let $x_n$ be a point from  $[0;1]$ such that $x_n=\hat \varphi^n (x_0)$.

Since $i_1 \in N^0 _{m_1}$, $i_2 \in N^0 _{m_2}$, ..., $i_n \in N^0 _{m_n}$ and the system \eqref{def: function eq. system - nega-tilde Q 1} is true and the condition 
$\overline{T}\left(\hat{\varphi}^{n}(x_0);F\left(\hat{\varphi}^{n}(x_0)\right)\right)\in~\Gamma_{\tilde{F}}$ holds, it follows that 
$$
\psi_{i_n,n}\circ ...\circ \psi_{i_2,2} \circ \psi_{i_1,1}\left(\overline{T}\right)=T_0(x_0;F(x_0))\in \Gamma_{F}, ~~~i_n \in N^0 _{m_n},~~~n\to \infty.
$$

Whence, $\Gamma_{F}\subset G$. So,
$$
\Gamma_{F}=\bigcup_{x \in [0;1]}{( ... \circ \psi_{i_n(x),n}\circ ...\circ \psi_{i_2(x),2} \circ \psi_{i_1(x),1}(\Gamma_F))}.
$$
\end{proof}

\end{document}